%% file: main.tex
\documentclass[a4paper,12pt]{amsart}
\usepackage[T1]{fontenc}
\usepackage[english]{babel}
\usepackage{amsthm, amssymb, amsfonts}
\usepackage{dsfont}
\usepackage{mathtools}
\usepackage{mathrsfs}
\usepackage[margin=1in]{geometry}
\usepackage{enumerate}
\usepackage{biblatex}
\renewbibmacro{in:}{}
\bibliography{references}
\usepackage{csquotes}
\usepackage{xcolor}
\usepackage{graphicx}

\usepackage{esint}
\numberwithin{equation}{section}
\theoremstyle{plain}
\newtheorem{theorem}{Theorem}[section]
\newtheorem{proposition}[theorem]{Proposition}
\newtheorem{corollary}[theorem]{Corollary}
\newtheorem{lemma}[theorem]{Lemma}
\theoremstyle{remark}
\newtheorem{remark}[theorem]{Remark}

\theoremstyle{definition}
\newtheorem{definition}[theorem]{Definition}

\DeclarePairedDelimiter{\abs}{\lvert}{\rvert}

\newcommand{\sgn}{\operatorname{sgn}}
\newcommand{\vol}{\operatorname{vol}}
\newcommand{\arctanh}{\operatorname{artanh}}

\newcommand{\Haus}{\mathscr{H}}

\newcommand{\Ham}{\mathbb{H}}

\newcommand{\Hc}{\textbf{H}}

\newcommand{\K}{\mathbb{K}}

\newcommand{\R}{\mathbb{R}}
\newcommand{\C}{\mathbb{C}}
\newcommand{\Oct}{\mathbb{O}}
\newcommand{\N}{\mathbb{N}}

\newcommand{\Dist}{\mathcal{H}}
\newcommand{\Vist}{\mathcal{V}}

\title[Approaching the isoperimetric problem in $ H_{\C}^m$]{Approaching the isoperimetric problem in $H_{\C}^m$ via the hyperbolic log-convex density conjecture}
\author{Lauro Silini}
\begin{document}
\begin{abstract}We prove that geodesic balls centered at some base point are isoperimetric in the real hyperbolic space $H_{\R}^n$ endowed with a smooth, radial, strictly log-convex density on the volume and perimeter. This is an analogue of the result by G. R. Chambers for log-convex densities on $\R^n$. As an application we prove that in any rank one symmetric space of non-compact type, geodesic balls are isoperimetric in a class of sets enjoying a suitable notion of radial symmetry.
\end{abstract}
\maketitle
\section{Introduction}
We denote by $(H_{\R}^n,g_{H})$ the real hyperbolic space of dimension $n\in\N$ with constant sectional curvature equal to $-1$. Call $d_H$ the induced Riemannian distance. Choose an arbitrary base point $o\in H_{\R}^n$. We say that a function $f:H_{\R}^n\to\R_{>0}$ is (strictly) radially log-convex if
\[
\ln(f(x))=h(d_{H}(o,x)),
\]
for a smooth, (strictly) convex and even function $h:\R\to\R$. We define the weighted perimeter and volume of a  set with finite  perimeter $E\subset H_{\R}^n$ as
\[
V_f(E):=\int_E f\,d\Haus^n,\quad\text{and}\quad P_f(E)=\int_{\partial^*E}f\,d\Haus^{n-1}.
\]
Here, following the notation in \cite{Maggi}, $\partial^*E$ denotes the reduced boundary of $E$. A set of finite perimeter $E$ with volume $V_f(E)=v>0$ is called isoperimetric  if it solves the minimization problem
\begin{equation}\label{eq:isoperimetric_inf}
\inf\Bigl\{P_f(F):V_f(F)=v,F\subset H_{\R}^n\text{ of finite perimeter}\Bigr\}.
\end{equation}
The first goal of this paper is to show the following characterization of the isoperimetric sets, which will be developed in Section \ref{sec:proof}.
\begin{theorem}\label{thm:log_convex}For any strictly radially log-convex density $f$, geodesic balls centered at $o\in H_{\R}^n$ are isoperimetric sets with respect to the weighted volume and perimeter $V_f$ and $P_f$.
\end{theorem}
This result is the hyperbolic twin of an analogous result in the Euclidean space, conjectured first by Kenneth Brakke (see \cite[Conjecture 3.12]{rosales2008isoperimetric}), stating that Euclidean balls centered at the origin are isoperimetric in $\R^n$ endowed with a log-convex density. We  invite the reader to consult F. Morgan \cite{morgan2003regularity} and F. Morgan and A. Pratelli \cite{morgan2013existence} for existence and regularity properties, V. Bayle, A. Cañete, F. Morgan and C. Rosales \cite{rosales2008isoperimetric} for the stability, A. Figalli and F. Maggi \cite{figalli2013isoperimetric} for the small volume regime, A. Kolesnikov and R. Zhdanov \cite{Kolesnikov} for the large volume regime. The conjecture was then recently completely solved by the remarkable article by G. R. Chambers \cite{chambers}. In fact, the first and main part of this paper is an adaptation of the method presented in the latter to the real hyperbolic space.
\par In the hyperbolic setting, the two dimensional case was solved by I. McGillivray in \cite{McGillivray}. We refer as well to \cite{Bongiovanni,Giosia,Howe,morgan2000isoperimetric} for other works related to this problem.
\par Even if interesting by itself, our main motivation in proving such result is the tight relation of this problem with the (unweighted) isoperimetric problem in the complex hyperbolic spaces $H_{\C}^m$, the quaternionic spaces $H_{\Ham}^m$ and the Cayley plane $H_\Oct^2$ restricted to a family of sets sharing a particular symmetry that we define as follows.
\begin{definition}[Hopf-symmetric sets]\label{def:hopf_symmetric}Let $\K\in\{\C,\Ham,\Oct\}$, $d=\dim(\K)\in\{2,4,8\}$ and $(M,g)=(H_{\K}^m,g)$ be the associated rank one symmetric space of non-compact type of  real dimension $n=dm$, $m=2$ if $\K=\Oct$. Fix an arbitrary point $o\in M$ and let $N$ be the unit length radial vector field emanating from $o$. Then, up to renormalization of the metric, the Jacobi operator $R(\cdot,N)N$ arising from the Riemannian curvature tensor $R$ is a self adjoint operator of $TM$, and has exactly three eigenvalues: 
$\{0,-1,-4\}$. The $(-4)$-eigenspace defines at every point $x\neq o$ a distribution $\Dist_x$ of real dimension $d-1$.  A $C^1$-set $E\subset M$ with normal vector field $\nu$ is said to be \emph{Hopf-symmetric} if $\nu(x)$ is orthogonal to $\Dist_x$ at each point $x\in\partial E$, $o\not\in\partial E$.
\end{definition}
\begin{remark}Let $h:S^{n-1}\to\K P^{m-1}$ be the celebrated Hopf fibration. Then, for any $C^1$-profile $\rho:S^{n-1}\to (0,+\infty)$ so that $\rho$ is constant along the fibres of $h$, the set with boundary
\[
\partial E:=\{\exp_o(\rho(x)):x\in S^{n-1}\subset T_oM\},
\]
is Hopf-symmetric.
\end{remark}
\begin{remark} Being \emph{Hopf-symmetric} has not to be confused with the standard notion of being \emph{Hopf} in $H_{\C}^m$, that is a set with principal curvature along the characteristic directions $J\nu$, where $J$ denotes the associated complex structure. It is worth saying that spheres are the only Hopf, compact, embedded mean curvature surfaces in $H_{\C}^m$, as it is proven by X. Wang in \cite{wang2017integral}. The natural generalization of this concept when $\K\in\{\Ham,\Oct\}$ is being a \emph{curvature-adapted hypersurface}, that is the normal Jacobi operator $R(\cdot,\nu)\nu$ commutes with the shape operator.
\end{remark}
We adopt the notation of Definition \ref{def:hopf_symmetric} for the rest of the paper. Let $P$ and $V$ be the perimeter and volume functionals induced by $g$ in $H_{\K}^m$. Consider the (unweighted) isoperimetric problem
\begin{equation}\label{eq:isoperimetric_inf_KH}
\inf\Bigl\{P(F):V(F)=v,F\subset H_{\K}^m\text{ Hopf-symmetric}\Bigr\}.
\end{equation}
We dedicate Section \ref{sec:corollary} to the proof of this theorem.
\begin{theorem}\label{thm:2}If geodesic balls centered at $o\in H_{\R}^n$ are isoperimetric with respect to Problem \eqref{eq:isoperimetric_inf} for the strictly radial log-convex density
\[
f(x)=\cosh(d_{H_{\R}^n}(o,x)))^{d-1},\quad d=\dim(\K),
\]
then geodesic balls in $H_{\K}^m$ are optimal with respect to the isoperimetic Problem \eqref{eq:isoperimetric_inf_KH}.
\end{theorem}
Combining this with Theorem \ref{thm:log_convex} we get immediately the following Corollary.
\begin{corollary}\label{cor}In the class of Hopf-symmetric sets, geodesic balls are isoperimetric regions in $H_{\K}^m$.
\end{corollary}
\subsection*{Acknowledgments}
The author would like to thank Urs Lang and Alessio Figalli for their precious guidance and constant support. A special thanks to Raphael Appenzeller for the very enjoyable and instructive exchanges about the hyperbolic plane. Lastly, the author would like to thank Miguel Dom\'inguez V\'azquez for suggesting an efficient way to include the octonionic case in the definition of Hopf-symmetric, and Frank Morgan, for the useful insight about the earlier contributions to the problem. The author has received funding from the European Research Council under the Grant Agreement No. 721675 “Regularity and Stability in Partial Differential Equations (RSPDE)”.
\section{Preliminaries}In what follows, we will always assume $E\subset H_{\R}^n$ to be an isoperimetric set with respect to the weighted problem \eqref{eq:isoperimetric_inf}.
\subsection{Qualitative properties of the isoperimetric sets}
Supposing $E$ smooth, the volume preserving first variation of the perimeter gives us that
\begin{equation}\label{eq:conserved_mean_curvature}
\Hc_f:=H+\partial_\nu h=\text{constant},
\end{equation}
on $\partial E$, where $H$ is the unaveraged inward mean curvature and $\nu$ is the  outward pointing unit normal.
Existence and regularity properties of isoperimetric sets are summarized in the following theorem. We refer to the papers of F. Morgan \cite{morgan2003regularity} and F. Morgan and A. Pratelli \cite{morgan2013existence}. Their results on $\R^n$ generalise directly to $H_{\R}^n$.
\begin{theorem}[Existence and regularity]\label{thm:regularity}For any volume $v>0$ there exists a set $E\subset H_{\R}^n$ of finite perimeter and weighted volume $V_f(E)=v$ solving the isoperimetric problem \eqref{eq:isoperimetric_inf}. Moreover, $E$ enjoys the following properties:
\begin{itemize}
\item[--]$\partial E$ is a bounded embedded hypersurface  with singular set of Hausdorff dimension at most $n-8$.
\item[--]There exists $\lambda\in\R$ such that at any regular point $x\in\partial E$, $\Hc_f(x)=\lambda$. As a consequence, $\partial E$ is mean-convex at each regular point $y\in\partial E$, that is $H(y)\geq (n-1)$.
\item[--]If  the tangent cone at $x\in\partial E$ lies in an halfspace, then it is an hyperplane, and therefore $\partial E$ is regular at $x$. In particular, $\partial E$ is regular at points $x^\star\in\partial E$ satisfying  $d_H(x^\star,o)=\sup_{x\in\partial E}d_H(x,o)$.
\end{itemize}
\end{theorem}
\subsection{The Poincaré model of $H_{\R}^n$}
Adopting the Poincaré model, $H_{\R}^n$ is conformal to the open Euclidean unit ball. At a point $x\in H_{\R}^n$ the metric is
\[
g_{H}=\frac{4}{(1-r^2)^2}g_{\text{flat}},
\]
where $r=\abs{x}$ will always denote the Euclidean distance of $x$ from the origin, and $g_{\text{flat}}$ the usual Euclidean metric of $\R^n$. The hyperbolic distance from the origin is then given by
\[
d_{H}(x,0)=2\arctanh(r).
\]
We define the boundary at infinity $\partial_\infty H_{\R}^n$ of $H_{\R}^n$ to be the Euclidean unit sphere $\partial B(0,1)=S^{n-1}$. We will identify the base point $o\in H_{\R}^n$ of the radial density $f$ with the origin $0$ in $B(0,1)$.
\subsection{Isometries and special frames in $ H_{\R}^2$}
Denote by $e_1$ and $e_2$ the horizontal and vertical Cartesian axis in the two dimensional Poincaré disk model. Also, let $(H_{\R}^2)_+$ be the intersection of $ H_{\R}^2$ with the closed upper half-plane having $e_1$ as boundary. The isometry group of $( H_{\R}^2,g_{H})$ is completely determined (up to orientation) by the Möbius transformations fixing the boundary $\partial_\infty H_{\R}^2$. Hence, geodesic spheres coincide with Euclidean circles completely contained in $ H_{\R}^2$. Their curvature lies in $(1,+\infty)$. Circles touching $\partial_\infty H_{\R}^2$ in a point are called horospheres, and have curvature equal to  1. Geodesics are arcs of (possibly degenerate) circles hitting $\partial_\infty  H_{\R}^2$ perpendicularly in two points. Arcs of   (possibly degenerate) circles that are not geodesics are called hypercycles, and have constant curvature in $(-1,1)\setminus\{0\}$. It will be convenient to work with a particular frame: define
\[
S: (H_{\R}^2)_+\to\R,
\]
to be the hyperbolic distance of a point in $ (H_{\R}^2)_+$ from the horizontal axis $e_1$. Set $X=\nabla S$, where we naturally extend by continuity $X$ at $e_1$. Then, denoting with $X^\perp$ the counterclockwise rotation of $X$ by $\frac\pi 2$ radians, since the level sets of $S$ are equidistant to each other, $\{X,X^\perp\}$ forms an orthonormal frame of $ (H_{\R}^2)_+$, see  Figure \ref{fig:frame}. The integral curves of $X$ are all geodesic rays hitting $e_1$ perpendicularly. For each $l\in[0,1)$, let $\delta_l$ be the integral curve of $X^\perp$ so that $\delta_l(0)=(0,l)$. Then, ${(\delta_l)}_{l\in[0,1)}$ is a family of equidistant hypercycles foliating $(H_{\R}^2)_+$, crossing $e_2$ perpendicularly and with constant curvature which coincides with the Eucliden one: $K_1=\frac{2l}{1+l^2}=\frac{1}{R(l)}$, where $R(l)\in(0,+\infty]$ is the radius of the Euclidean circle representing the curve. Similarly, set $\{N,N^\perp\}$ to be the orthonormal frame on $ H_{\R}^2\setminus\{0\}$ where $N$ is the radial unit length vector field emanating from the origin. Then, integral curves of $N$ are rays of geodesics, and integral curves of $N^\perp$ are concentric geodesic spheres.  Notice that the frame $\{X,X^\perp\}$ is invariant under the one-parameter subgroup of hyperbolic isometries fixing $e_1$  ($X^\perp$ is the infinitesimal generator of the action by translations) and, up to orientation reversing, under the reflections with respect to any geodesic  integral curve of $X$. Finally, notice that on $e_1$  and $e_2$, $\{X,X^\perp\}$ is a positive rescaling of $\{(0,1),(-1,0)\}$.
\begin{figure}[htbp]
\centering
\includegraphics[scale=0.7]{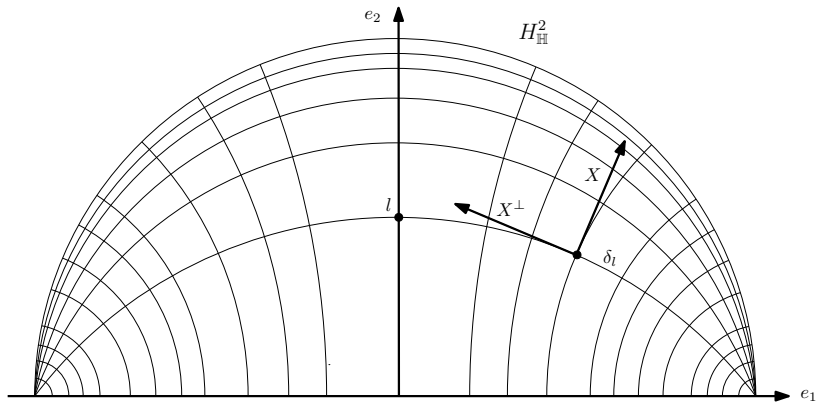}
\caption{The special frame $\{X,X^\perp\}$.}
\label{fig:frame}
\end{figure}
\par 
For a regular curve parametrized by arc length $\eta$ we denote  with $\kappa_\eta(t)$ the inward signed curvature of $\eta$ at $\eta(t)$. We recall the identity
\[
\kappa_\eta\dot\eta^\perp=\nabla_{\dot\eta}\dot\eta,
\]
where here $\nabla$ denotes the standard Levi-Civita connection associated to $g_{H}$.
\subsection{Reduction to $ H_{\R}^2$}From now on, let $E$ be an isoperimetric set with arbitrary weighted volume.
Since both the density $f$ and the conformal term  of $g_{H}$ are radial, the coarea formula implies that spherical symmetrization  pointed at the origin preserves the weighted volume and does not increase the weighted perimeter of $E$.  For this reason, we will assume $E$ spherically symmetric with respect to the $e_1$ axis. Now, intersecting $E$ with the the Euclidean plane spanned by $\{e_1,e_2\}$, we obtain a spherically symmetric profile $\Omega\subset  H_{\R}^2$. Let $x^\star$  be the  furthest point of $\Omega$ lying in the  positive part of the $e_1$ axis (this is always possible by reflecting $\Omega$ with  respect to the $e_2$ geodesic). Let $\gamma:[-a,a]\to  H_{\R}^2$ be a counter-clockwise, arclength parametrization of the boundary of the connected component of $\Omega$ containing $x^\star$, so that  $\gamma(0)=x^\star$, see Figure \ref{fig:sym}. The curve $\gamma$ enjoys the following properties:
\begin{itemize}
\item[--]$\gamma$ is smooth on $(-a,a)$. Indeed, if there exists $a^*\in(-a,a)$ such that $\gamma(a^*)$ is not regular, then $\partial E$ contains a singular set of Hausdorff dimension $n-2$, but this cannot be because of Theorem \ref{thm:regularity}.
\item[--]The curve $\gamma$ forms a simple, closed curve.
\item[--]Writing $\gamma=(\gamma_1,\gamma_2)$ in cartesian coordinates, one has that $\sgn(\gamma_2(t))=\sgn(t)$. In particular, $\gamma:[0,a)\to (H_{\R}^2)_+$.
\item[--]$\dot\gamma(0)=X(\gamma(0))$.
\end{itemize}

\begin{figure}[htbp]
\centering
\includegraphics[scale=0.7]{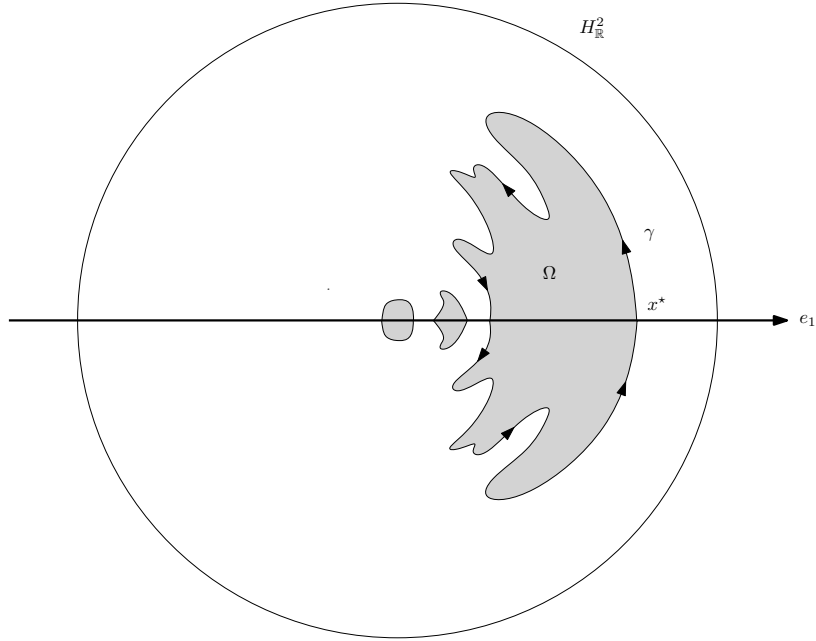}
\caption{The spherical symmetrization.}
\label{fig:sym}
\end{figure}
To translate Equation \eqref{eq:conserved_mean_curvature} as a property  of  the profile $\gamma$, we need the following definition.
\begin{definition}\label{def:comparison_c}For any $t\in[0,a)$, denote by
\begin{itemize}
\item[--]$C_t=C_t(s)$ the (possibly degenerated) oriented circle tangent to $\gamma(t)$, with center on $e_1$, parametrized by arclength and such that $C_t(0)=\gamma(t)$. Denote by $\kappa(C_t)$ its constant curvature. 
\item[--]Similarly, call $c_t=c_t(s)$ the (possibly degenerated) oriented circle tangent to $\gamma(t)$, parametrized by arclength, such that $c_t(0)=\gamma(t)$ and $\kappa(c_t)=\kappa_\gamma(t)$. 
\item[--]Define $x(C_t)$ and $x(c_t)$ to be the hyperbolic center of $C_t$ and $c_t$ respectively. Similarly, let $x_1(C_t)$ and $x_1(c_t)$ be the first Euclidean coordinate of $x(C_t)$ and $x(c_t)$ respectively.
\end{itemize}
\end{definition}
\begin{remark}\label{rmk:conformal}Let $F\subset B(0,1)\subset \R^n$. Then, at every regular point $x\in\partial F$, the mean curvature $H$ is related with the Euclidean mean curvature $H^\text{flat}$ by
\[
H(x)=\frac{1-r^2}{2}H^\text{flat}(x)+(n-1)g_{\text{flat}}(x,\tilde\nu),
\]
where $\tilde\nu$ is the outward normal vector to $\partial F$ with Euclidean norm equal to one.
In particular, when $n=2$, denoting with $\kappa^{\text{flat}}$ the usual Euclidean curvature, one has that
\[
\kappa_\eta=\frac{1-\abs{\gamma(t)}^2}{2}\kappa^{\text{flat}}_\eta+g_{\text{flat}}(\eta,\tilde\nu),
\]
Therefore,  $\kappa^{\text{flat}}(c_t)=\kappa^{\text{flat}}_\gamma$, that is comparison circles $c_t$ and $C_t$ in the hyperbolic setting coincide with  comparison circles with respect to the Euclidean metric. From this formula, we also deduce that for any Euclidean circle $\mathscr{C}$
\[
\kappa_{\mathscr{C}}=\frac{1}{2}\Bigl(\frac{1-\abs{x_0}^2}{\tau_0}+\tau_0\Bigr)=\coth(\tau),
\]
where $x_0$ and $\tau_0$ are the Euclidean center and radius, and $\tau$ is the hyperbolic radius.
\end{remark}
\begin{lemma}\label{lem:const}On $t\in[0,a)$ it holds
\[
H=\kappa_\gamma+(n-2)\kappa(C_t).
\]
In particular,
\[
\Hc_f=\kappa_\gamma+(n-2)\kappa(C_t)+h'g_{H}(\nu,N)=\lambda,
\]
where $\nu=-\dot\gamma^\perp$.
\end{lemma}
We call $H_1:=h'g_{H}(\nu,N)$  the term coming from the log-convex density.
\begin{proof}In \cite[Proposition 3.1]{chambers} it is shown that the Euclidean mean curvature of the spherically symmetric set $E$ can be computed as
\[
H^\text{flat}=\kappa_\gamma^\text{flat}+(n-2)\kappa^\text{flat}(C_t).
\]
Thanks to Remark \ref{rmk:conformal} we have that
\begin{align*}
H(\gamma(t))&=\frac{1-\abs{\gamma(t)}^2}{2}H^\text{flat}(\gamma(t))+(n-1)g_{\text{flat}}(\gamma(t),\nu)\\
&=\frac{1-\abs{\gamma(t)}^2}{2}\Bigl(\kappa_\gamma^\text{flat}+(n-2)\kappa^\text{flat}(C_t)\Bigr)+(n-1)g_{\text{flat}}(\gamma(t),\nu)\\
&=\kappa_\gamma+(n-1)\kappa(C_t).
\end{align*}
\end{proof}
\section{The proof}\label{sec:proof}The reduction to $ H_{\R}^2$ implies that Theorem \ref{thm:log_convex} is equivalent to showing that $\gamma$ represents a circumference centered in the origin. Essentially, we prove that ruling out this possibility, implies that $\gamma$ has to make a curl (see Figure \ref{fig:curl}), contradicting the fact that $\gamma$ parametrizes a spherically symmetric set. This is made rigorous by the combination of the next two lemmas.
\begin{lemma}\label{lem:spherically_monotonicity_tangent}For every $t\in(0,a)$
\[
g_{H}(N,\dot\gamma)\leq 0.
\]
\end{lemma}
\begin{proof}The set $\Omega$ spherically symmetric implies that $t\mapsto g_{\text{flat}}(\gamma(t),\gamma(t))$ is non increasing. Differentiating in $t$ gives the desired sign of the angle between $N$ and $\dot\gamma$.
\end{proof}
Section \ref{sec:tangent} is devoted to the proof of the next lemma. 
\begin{lemma}[The tangent lemma]\label{lem:tangent}If $\gamma$ is not a circle centered in the origin, there exist $0<a_0<a_1<a_2<a$ such that $\dot\gamma(a_0)=X^\perp(\gamma(a_0))$, $\dot\gamma(a_1)=-X(\gamma(a_1))$ and $\dot\gamma(a_2)=X(\gamma(a_2))$.
\end{lemma}
\begin{figure}[htbp]
\centering
\includegraphics[scale=0.7]{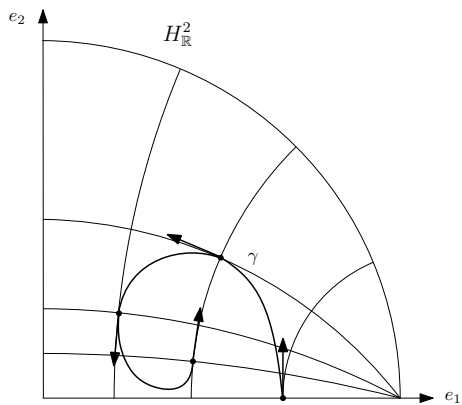}
\caption{The curl described in Lemma \ref{lem:tangent}.}
\label{fig:curl}
\end{figure}
Assuming now that Lemma \ref{lem:tangent} holds true, the proof of the main result goes as follows.
\begin{proof}[Proof  of Theorem \ref{thm:log_convex}]If $\gamma$ is a circle centered at the origin we are done. Otherwise, Lemma \ref{lem:tangent} ensures the existence of $0<a_2<a$ such that $\dot\gamma(a_2)=X(\gamma(a_2))$. This violates the inequality of Lemma \ref{lem:spherically_monotonicity_tangent}, because at $a_2$
\[
g_{H}(N,\dot\gamma)=g_{H}(N,X)>0.
\]
\end{proof}
\subsection{Proof of the tangent lemma}\label{sec:tangent}
\input{tangent_lemma}
\section{Symmetric sets in $H_{\K}^m$}\label{sec:corollary}
\input{hopf_sets2}
\begingroup
\raggedright
\sloppy
\printbibliography
\end{document}

%% file: tangent_lemma.tex
The proof is made by following the behaviour of $\gamma$ step by step: first we show that $\gamma$ has to arch upwards with curvature strictly  greater than one. The endpoint of this arc will be $\gamma(a_0)$, where $\dot\gamma(a_0)=X^\perp(\gamma(a_0))$. Then, it goes down curving strictly faster than before, and this result about curvature comparison is the tricky point to  generalize in the  hyperbolic setting. It turns out that the  special frame given by the hypercyclical foliation $(\delta_l)_{l\in[0,1)}$ is the good one. Then, arguing by  contradiction, we will show that this behaviour must end at a point  $0<a_0<a_1<a$, where $\dot\gamma(a_1)=-X(\gamma(a_1))$. Finally, we prove the existence of $a_2$ so that $\dot\gamma(a_2)=X(\gamma(a_2))$ taking advantage of the mean-curvature convexity of $\Omega$. We start by looking at what happens at the  starting point.
\begin{lemma}\label{lem:first_dynamic}One has that $\dot\gamma(0)=X(\gamma(0))$, $\dot\kappa_\gamma(0)=0$ and $\kappa_\gamma(0)\geq\kappa(C_0)>1$.
\end{lemma}
\begin{proof}This is a consequence of the symmetry of $\gamma$ with respect to the $e_1$ axis, and that $\gamma(0)$ represent the furthest point from the origin of $\Omega$.
\end{proof}
\begin{lemma}\label{lem:uniqueness}If there exists $t^*\in[0,a)$ such that $x_1(C_{t^*})=0$ and $\kappa_\gamma(t^*)=\kappa(C_{t^*})$, then $\gamma$ has to be a centered circle.
\end{lemma}
\begin{proof}In this case $\gamma(t)$ and $C_{t^*}(s)$ solve the same ODE, with same initial data. Therefore, they have to coincide locally, and hence globally.
\end{proof}
\begin{definition}Call $\alpha:[0,a)\to [\pi,-\pi)$ the oriented angle made by $\dot\gamma$ with $X^\perp$. We say that $\dot\gamma(t)$ is in the I, II, III and IV quadrant if $\alpha(t)$ belongs to  $[\pi/2,\pi]$, $[0,\pi/2]$, $[0,-\pi/2]$ and $[-\pi/2,-\pi]$ respectively. We add \emph{strictly} if $\dot\gamma$ is not collinear to $X$ and $X^\perp$.
\end{definition}
\begin{lemma}\label{lem:center_C}If for $t\in[0,a)$, $\dot\gamma(t)$ belongs to the II quadrant, then $x_1(C_t)\geq 0$.
\end{lemma}
\begin{proof}
We first treat the case $t\in(0,a)$. Expressing $N(\gamma(t))$ in the $\{X,X^\perp\}$ frame, we have thanks to Lemma \ref{lem:spherically_monotonicity_tangent} that
\begin{equation}\label{eq:angles}
\begin{split}
0\geq g_H(N,\dot\gamma)&=g_H(X,N)g_H(X,\dot\gamma)+g_H(X^\perp,N)g_H(X^\perp,\dot\gamma)\\
&=g_H(X,N)\sin(\alpha)+g_H(X^\perp,N)\cos(\alpha).
\end{split}
\end{equation}
If $\alpha=\pi/2$, then 
\[
0\geq g_H(X,N),
\]
which is possible only when $\gamma_2(t)=0$, that is $t\in\{0,a\}$. If $\alpha\in [0,\pi/2)$ then $\cos(\alpha)>0$, implying that $g_H(X^\perp,N)\leq-g(X,N)\sin(\alpha)\leq 0$. Notice that this is possible only if $\gamma_1(t)\geq 0$. Calling $-\vartheta<0$ the angle that $N$ makes with $X^\perp$, we get by Equation \eqref{eq:angles} that
\begin{equation}\label{eq:ANGLES}
\tan(\alpha)\leq\tan(\vartheta).
\end{equation}
Now, observe that the two geodesic rays $\sigma_\gamma$, $\sigma_N$ stating at $\gamma(t)$ with initial velocities $\dot\sigma_\gamma(0)=\dot\gamma^\perp(t)$ and $\dot\sigma_N(0)=N^\perp(\gamma(t))$, together with the axis $e_1$ and the geodesic orthogonal to $e_1$ passing from $\gamma(t)$ bound two geodesic triangles $\triangle_\gamma$ and $\triangle_N$. Call $d$ the distance between $\gamma(t)$ and $e_1$. Then, the length of the sides $\ell_\gamma$ and $\ell_N$ of $\triangle_\gamma$ and $\triangle_N$ respectively, lying on $e_1$ are given via hyperbolic trigonometric laws by
\[
\tanh(\ell_\gamma)=\tan(\alpha)\sinh(d)\leq\tan(\vartheta)\sinh(d)=\tanh(\ell_N).
\]
But this implies that $x(C_t)$, which is the intersection of $\sigma_\gamma$ with $e_1$, has first coordinate positive, as claimed.
If $t=0$, then $C_0=c_0$ and approximates $\gamma(0)$ up to the fourth order. Therefore, if $x_1(C_t)<0$, then there exists $\varepsilon>0$ such that $\gamma\lvert_{(\varepsilon,2\varepsilon)}$ lies outside the ball of centered in the origin and with radius $d_H(\gamma(0),o)$. This is a contradiction because by construction $\gamma(0)$ is the furthest point of $\Omega$ from $o$.
\end{proof}
Our next goal is to show four important properties of the curve $\gamma$. The proof is made by comparison with the circles $c_t$ and $C_t$, and the preservation of the weighted mean curvature $\Hc_f$. For this reason, we need the following preliminary lemma.
\begin{lemma}\label{lem:H1_circle}Let $\eta=\eta(s)$ be an arc-length, counter-clockwise parametrization of a circle centered at $(0,y)$ such that $\eta(0)=(\tau,y)$ and $\eta(L)=(0,y+\tau)$. Let $O=(-\tilde o,0)$ be an arbitrary point lying on $e_1$ with $\tilde o\in [0,1)$, and $\nu(s)$ the outward pointing normal to $\eta(s)$. Then, setting
\[
\tilde H_1(s):=\partial_\nu (h(d_{H}(\eta(s),O))),
\]
if $y=0$ then
\begin{equation}\label{eq:sgn_H1}
\tilde H_1'(s)\leq 0,\text{ in }(0,L],
\end{equation}
and
\begin{equation}\label{eq:sgn_H12}
\tilde H''_1(0)\leq 0.
\end{equation}
Both inequalities are strict if $\tilde o\neq 0$.
If $y\in(0,1)$ and $\tilde o\neq 0$, then
\begin{equation}\label{eq:sgn_H123}
\tilde H'_1(L)<0.
\end{equation}
\end{lemma}
\begin{proof}Let $T: H_{\R}^2\to H_{\R}^2$ be the unique isometry fixing $e_1$ and sending the origin to $O$. Then,
\begin{align*}
\tilde H'_1(s)&=h''g_H(T_* N,\dot\eta)g_H(\nu,T_*N)+h'\frac{d}{ds}g_H(\nu,T_*N)\\
&=h''g_H(T_* N,\dot\eta)g_H(\nu,T_*N)-h'\Bigl(g_H(\nabla_{\dot\eta}\dot\eta^\perp,T_*N)+g_H(\dot\eta^\perp,\nabla_{\dot\eta}T_*N)\Bigr)\\
&=h''g_H(T_* N,\dot\eta)g_H(\nu,T_*N)-h'\Bigl(-g_H(T_*N,\dot\eta)\kappa_\eta+g_H(T_*N,\dot\eta)g_H(T_*N^\perp,\dot\eta)\kappa_1\Bigr),
\end{align*}
where $\kappa_1$ is the curvature of the integral curve of $T_*N^\perp$ passing through $\eta(s)$, which is a geodesic sphere centered at $O$. Suppose first that $y=0$ and $\tilde o\neq 0$. Then, $\dot\eta=N^\perp$, and
\[
\tilde H_1'(s)=h''g_H(T_*N,N^\perp)g_H(T_*N,N)-h'g_H(T_*N,N^\perp)(-\kappa_\eta+g_H(T_*N,N)\kappa_1)<0,
\]
because $h''>0$, $h'>0$, $g_H(T_*N,N^\perp)<0$, $g_H(T_*N,N)>0$  and $\kappa_\eta>\kappa_1$ since $\tilde o\neq 0$.
This proves Equation \eqref{eq:sgn_H1} when $\tilde o\neq 0$. The same holds in the context of Equation \eqref{eq:sgn_H123} since $\dot\eta(L)=N^\perp$. Up to relaxing the inequalities, the proof when $\tilde o=0$ is exactly the same.
To prove Equation \eqref{eq:sgn_H12}, we differentiate $\tilde H_1$ one more time, obtaining
\begin{align*}
\tilde H''_1(s)&=h'''g_H(T_* N,\dot\eta)^2g_H(\nu,T_*N)+h''\frac{d}{ds}g_H(T_*N,\dot\eta)g_H(\nu,T_*N)\\
&\quad+h''g_H(T_*N,\dot\eta)\frac{d}{ds}g_H(\nu,T_*N)+h''g_H(T_*N,\dot\eta)\frac{d}{ds}g_H(\nu,T_*N)\\
&\quad+h'\frac{d^2}{ds^2}g_H(\nu,T_*N).
\end{align*}
Observe that in zero $g_H(T_*N,\dot\eta)=0$, hence only the second and last term survive
\[
\tilde H_1''(0)=h''\frac{d}{ds}\Big\lvert_{s=0}g_H(T_*N,N^\perp)g_H(T_*N,N)+ h'\frac{d^2}{ds^2}\Big\lvert_{s=0}g_H(T_*N,N).
\]
Taking advantage of the explicit expression for $\frac{d}{ds}g(T_*N,N)$ we obtained before, we get
\begin{align*}
\frac{d^2}{ds^2}\Big\lvert_{s=0}g_H(T_*N,N)&=-\frac{d}{ds}\Big\lvert_{s=0}\Bigl(g_H(T_*N,N^\perp)\bigl(-\kappa_\eta+g_H(T_*N,N)\kappa_1\bigr)\Bigr)\\
&=\bigl(\kappa_\eta-g_H(T_*N,N)\kappa_1\bigr)\frac{d}{ds}\Big\lvert_{s=0}g_H(T_*N,N^\perp),
\end{align*}
which implies that
\[
\tilde H''_1(0)=\underbrace{\bigl(h''g_H(T_*N,N)+h'\kappa_\eta-h'g_H(T_*N,N)\kappa_1\bigr)}_{> 0}\frac{d}{ds}\Big\lvert_{s=0}g_H(T_*N,N^\perp).
\]
Hence, we are left to show that $\frac{d}{ds}\Big\lvert_{s=0}g_H(T_*N,N^\perp)<0$. Developing again we get
\begin{align*}
\frac{d}{ds}\Big\lvert_{s=0}g_H(T_*N,N^\perp)&=g_H(\nabla_{N^\perp}T_*N,N^\perp)+g_H(T_*N,\nabla_{N^\perp}N^\perp)\lvert_{s=0}\\
&=g_H(T_*N,N)^2\kappa_1\lvert_{s=0}-g_H(TN,N)\kappa_\eta\lvert_{s=0}\\
&=\kappa_1-\kappa_\eta<0.
\end{align*} 
\end{proof}
We are now ready to prove the next result.
\begin{lemma}\label{lem:preliminary_prop}The following four points hold.
\begin{enumerate}[i.]
\item\label{itm:center_C}If for $t\in(0,a)$ one has that $\kappa_\gamma(t)\geq\kappa(C_t)>1$, then $t\mapsto x_1(C_t)$ is smooth and $\frac{d}{dt}x_1(C_t)\geq 0$.
\item\label{itm:existence_I_U} If $\gamma$ is not a centered circle, then $\ddot\kappa_\gamma(0)>0$.
\item \label{itm:_U_end} If for $t\in(0,a)$, $\dot\gamma(t)$ is in the II quadrant and $\kappa_\gamma(t)=\kappa(C_t)>1$, then $\dot\kappa_\gamma(t)\geq0$. Moreover, if $\dot\gamma(t)\neq X^\perp(\gamma(t))$ and $C_t$ is not centered in the origin, then $\dot\kappa_\gamma(t)>0$.
\item \label{itm:I_L_big}If for $t\in(0,a)$ one has that $\dot\gamma(t)=X^\perp(\gamma(t))$, $\gamma_1(t)>0$ and $\kappa_\gamma(t)\geq\kappa(C_t)>1$,  then $\dot\kappa_\gamma(t)>0$.
\end{enumerate}
\end{lemma}
\begin{proof}
We start with point \ref{itm:center_C}. Observe that since $c_t$ approximates $\gamma$ up to the third order around $\gamma(t)$, it suffices to prove $\frac{d}{dt}x_1(C_t)\geq 0$ replacing $\gamma$ with $c_t$. Also, we can suppose $x(c_t)$ on $e_2$ by composing with the unique hyperbolic isometry translating $x(c_t)$ on $e_2$ and fixing $e_1$. The curvature condition $\kappa_\gamma(t)\geq\kappa(C_t)>1$ ensures that $x(c_t)\in (H^2_{\R})_+$. By monotonicity of the function $\tanh(\cdot/2)$, it suffices to prove the claim for the Euclidean center of $C_t$. Thus, we have reduced the problem to an explicit computation in the Euclidean plane, that can be found in \cite[Lemma 5.3]{chambers}. Thanks to Lemma \ref{lem:H1_circle} the proofs of the other points go exactly as in \cite[Lemma 3.4, Lemma 3.5 and Lemma 3.7]{chambers}. We show point \ref{itm:existence_I_U}. Differentiating $\Hc_f$ twice, we get that
\[
\ddot\kappa_\gamma(0)=-(n-2)\ddot\kappa(C_0)-H_1''.
\]
By symmetry, $c_0=C_0$. Moreover, since $\dot\kappa_\gamma(0)=0$, we have that both $c_0$ and $C_0$ approximate $\gamma$ up to the fourth order near zero. Hence, $\ddot\kappa(C_0)=0$. Therefore, it suffices to determine the sign of $H_1''$ replacing $\gamma$ with $C_0$. Let $T: H_{\R}^2\to H_{\R}^2$ be the unique isometry fixing $e_1$ that moves $x(C_0)$ to the origin. The result follows by Equation \eqref{eq:sgn_H12} of Lemma \ref{lem:H1_circle} setting $O=T(0)$, and noticing that $T(0)\neq 0$ by Lemma \ref{lem:uniqueness}. The proofs of points \ref{itm:_U_end}. and \ref{itm:I_L_big}. are similar: in the first case the condition $\kappa_\gamma(t)=\kappa(C_t)$ implies that $c_t=C_t$ approximates $\gamma$ near $t$ up to the third order, the same holds if $\dot\gamma(t)=X^\perp(\gamma(t))$ by symmetry. Hence, substituting $\gamma$ with $c_t$ and differentiating one time $\Hc_f$, we have to determine the sign of $H_1'$ in the case of a circle, via Lemma \ref{lem:H1_circle}.
\end{proof}
We are now ready to analyse the first behaviour of $\gamma$.
\begin{definition}[Upper curve]\label{def:I_U}The upper curve is the (possibly empty) maximal connected interval $I_U\subset[0,a)$ such that $0\in I_C$ and for all $t\in I_U$
\begin{enumerate}[a.]
\item\label{itm:1_IU}$\dot\gamma(t)$ is in the II quadrant,
\item\label{itm:2_IU}$\kappa_\gamma(t)\geq\kappa(C_t)>1$,
\item\label{itm:3_IU}$t\mapsto x_1(C_t)$ is smooth and $\frac{d}{dt}x_1(C_t)\geq 0$.
\end{enumerate}
\end{definition}
We set
\[
a_0:=\sup I_U.
\]
In the discussion, we will sometimes identify the upper curve with its image through $\gamma$.
\begin{definition}\label{def:UC}We say that a curve $\eta$ is graphical with respect to the hypercyclic foliation $(\delta_l)_{l\in[0,1)}$ if $\eta$ meets each $\delta_l$ at most once.
\end{definition}
 Notice that the upper curve (if non empty) is graphical with respect to the hypercyclical foliation because $\dot\gamma$ is in the II quadrant
\begin{proposition}\label{prop:up}The upper curve is non empty and enjoys the following properties
\begin{enumerate}[i.]
\item$0<a_0<a$,
\item$a_0\in I_U$,
\item$\gamma_1(a_0)>0$,
\item$\dot\gamma(a_0)=X^\perp(\gamma(a_0))$.
\end{enumerate}
\end{proposition}
\begin{proof}Thanks to Lemma \ref{lem:preliminary_prop}, the proof goes exactly as \cite[Lemma 3.11 and Proposition 3.12]{chambers}. We sketch for completeness the idea behind each point. The upper curve is non empty because by Lemma \ref{lem:first_dynamic} and Lemma \ref{lem:preliminary_prop} points \ref{itm:center_C}. and \ref{itm:existence_I_U}. since $c_0=C_0$ approximates $\gamma$ up to the fourth order near zero, we have that there exists $\varepsilon>0$ such that in $[0,\varepsilon)$, $\dot\gamma$ belongs to the II quadrant, $\kappa_\gamma\geq\kappa(C_t)>1$ and $\frac{d}{dt}x_1(C_t)\geq 0$. Hence, $[0,\epsilon)\subset I_U$. Notice that $0<a_0$ cannot be equal to $a$ since otherwise the curve $\gamma$ does not close itself on $e_1$, simply because $\dot\gamma$ belongs to the II quadrant by definition of $I_U$. By the regularity of $\gamma$ and that $I_U$ is defined by three close conditions, we have that $a_0\in I_U$. By composing with the unique hyperbolic isometry sending $\gamma(a_0)$ on $e_2$ fixing $e_1$, we can see that $x_1(C_{a_0})\leq 0$ because $\dot\gamma(a_0)$ belongs to the II quadrant. Lemma \ref{lem:uniqueness} and Lemma \ref{lem:center_C} imply that $x_1(C_0)>0$ and since $\frac{d}{dt}x_1(C_t)\geq 0$ in $I_U$, one must have that $\gamma_1(a_0)>0$. The last point is proved by contradiction: if $\dot\gamma(a_0)\neq X^\perp(\gamma(a_0))$, then $a_0\in I_U$ implies that $\dot\gamma(a_0)$ is strictly in the II quadrant. If $\kappa_\gamma(a_0)=\kappa(C_{a_0})>1$, then $c_{a_0}=C_{a_0}$ approximates $\gamma$ to the third order and Lemma \ref{lem:preliminary_prop} point \ref{itm:_U_end}. implies that there exists some $\delta>0$ such that $\kappa_\gamma(t)\geq\kappa(C_t)>1$ for $t\in[a_0,a_0+\delta)$. The same holds if $\kappa_\gamma(a_0)>\kappa(C_{a_0})>1$ by continuity.  This means that $[a_0,a_0+\delta)\subset I_U$, which is not possible by the very definition of $a_0$. Hence, $\dot\gamma(a_0)=X^\perp(\gamma(a_0))$. 
\end{proof}
\begin{definition}[Lower curve]The lower curve  is the maximal connected interval $I_L\subset[a_0,a)$ such that for all $t\in I_L$
\begin{enumerate}[a.]
\item $a_0\in I_L$,
\item $\dot\gamma(t)$ is in the III  quadrant,
\item calling $\bar t\in I_U$ the unique time such that $S(\gamma(t))=S(\gamma(\bar t))$ we have that $\kappa_\gamma(t)\geq\kappa_\gamma(\bar t)$.
\end{enumerate}
We set
\[
a_1:=\sup I_L.
\]
\end{definition}
Notice that $a_0$ truly belongs to $I_L$, so $I_L\neq\{\}$. Also, the lower curve is graphical with respect to the hypercyclical foliation because $\dot\gamma$ is in the III quadrant. Our next goal is to prove that $a_1<a$. Again, we proceed by contradiction, and the intuition is the following: suppose that $a_1=a$. If $\kappa_\gamma(t)=\kappa_\gamma(\bar t)$  for all $t\in I_L$, then the lower curve is nothing else than the upper curve reflected with respect to the geodesic orthogonal to $e_1$ and passing through $\gamma(a_0)$. Hence, $\lim_{t\to a^+}\alpha(t)=-\frac\pi2$. Otherwise, if the $\gamma\lvert_{I_L}$ curves strictly faster than the upper curve at some point, then the angle of incidence $\lim_{t\to a^+}\alpha(t)<-\frac\pi 2$ (see Figure \ref{fig:curvature}). But this cannot be, because it contradicts the regularity of $\partial E$ points pointed out in Theorem \ref{thm:regularity}. To prove that the lower curve curves strictly faster than the upper curve we need first to express the curvature with respect to the $\{X,X^\perp\}$ frame, and next prove three comparison lemmas.
 \begin{lemma}\label{lem:formula_k}Let $\eta$ any regular curve parametrized by arclength. Then,
\[
-\kappa_\eta(t)=\dot\beta(t)-K_1(\eta(t))\cos(\beta(t)),
\]
where $\beta(t)$ denotes the angle between $\dot\eta$ and $X^\perp$, and $K_1$ is the curvature of the leaf $\delta_l$ passing through $\eta(t)$.
\end{lemma}
\begin{proof}Decompose $\dot\eta=AX+BX^\perp$. Then, since $\kappa_\eta\dot\eta^\perp=\nabla_{\dot\eta}\dot\eta$, we get that
\[
-\cos(\beta)\kappa_\eta=g_H\Bigl(\nabla_{\dot\eta}\dot\eta,X\Bigr)=\partial_t\bigl(\sin(\beta)\bigr)-g_H\Bigl(\dot\eta,\nabla_{\dot\beta}X\Bigr).
\]
Now, keeping in mind that $\nabla_XX=0$ and $g_{\Ham}\bigl(\nabla_{X^\perp}X^\perp,X\bigr)=-K_1(\eta(t))$, we get
\[
g_H\Bigl(AX+BX^{\perp},\nabla_{AX+BX^\perp}X\Bigr)=B^2g_H\Bigl(X^{\perp},\nabla_{X^\perp}X\Bigr)=\cos(\beta)^2K_1(\eta(t)).
\]
Dividing both sides by $\cos(\beta)$ we get the desired identity.
\end{proof}
We can prove our first curvature comparison lemma.
\begin{lemma}[$\kappa$ comparison lemma]\label{lem:k_comparison}Let $\eta_1:(0,A_1]\to H_{\R}^2$ and $\eta_2:(0,A_2]\to H_{\R}^2$ be two hypercyclical graphical curves parametrized by arclength and with velocity vectors in the II quadrant. Suppose that there exists $l_0\in[0,1)$ such that 
\[
\lim_{t\to 0^+}\eta_1(t)\text{ and }\lim_{t\to 0^+}\eta_2(t),
\]
exist and belong to the same leaf $\delta_{l_0}$. Also, suppose that  $\eta_1(A_1)=\eta_2(A_2)$, $\dot\eta_1(A_1)=\dot\eta_2(A_2)$, and that if $S(\eta_1(t))=S(\eta_2(\tau))$ then $\kappa_{\eta_1}(t)\geq\kappa_{\eta_2}(\tau)$. Then, calling $\alpha_1$ and $\alpha_2$ the angle made by $\dot\eta_1$ and $\dot\eta_2$ with $X^\perp$ we have that
\[
\lim_{t\to 0^+}\alpha_1(t)\geq \lim_{t\to 0^+}\alpha_2(t).
\]
Moreovoer, if for some $t$ and $\tau$ such that $S(\eta_1(t))=S(\eta_2(\tau))$, one has that $\kappa_1(t)>\kappa_2(\tau)$, then
\[
\lim_{t\to 0^+}\alpha_1(t)>\lim_{t\to 0^+}\alpha_2(t).
\]
\end{lemma}
\begin{proof}Since the curves are graphical with respect to the hypercyclical foliation we can operate a change of variable: we observe that the two height functions $s_1(t):=S(\eta_1(t))$ and $s_2(\tau)=S(\eta_2(\tau))$ are bijections with same image of the form $(l_0,L]$. By hypothesis $\kappa_{\eta_1}(s_1^{-1}(l))\geq\kappa_{\eta_2}(s_2^{-1}(l))$ for every $l\in(l_0,L]$. Comparing the two curves in the $l\in(l_0,L]$ variable, since $s_i'=g_H(\nabla S,\dot\eta_i)=g_H(X,\dot\eta_i)=\sin(\alpha_i)$, $i=1,2$, we get by Lemma \ref{lem:formula_k} that 
\[
0\leq\kappa_{\eta_1}(l)-\kappa_{\eta_2}(l)=\dot\alpha_2(l)\sin(\alpha_2(l))-\dot\alpha_1(l)\sin(\alpha_1(l))-\frac{2l}{1+l^2}\bigl(\cos(\alpha_2(l))-\cos(\alpha_1(l))\bigr).
\]
Multiplying by $(1+l^2)$ and integrating we finally get that
\begin{align*}
0&\leq\int_{l_0}^L(1+l^2)(\cos(\alpha_1)
\cos(\alpha_2))'+2l(\cos(\alpha_1)-\cos(\alpha_2))\,dl\\
&=\int_{l_0}^L\frac{d}{dl}\Bigl((1+l^2)(\cos(\alpha_1)-\cos(\alpha_2))\Bigr)\,dl=\lim_{l\to l_0^+}(1+l^2)(\cos(\alpha_2(l))-\cos(\alpha_1(l))).
\end{align*}
If the two curvatures are different somewhere, then the inequality between the two angles is strict.
\end{proof}
\begin{figure}[htbp]
\centering
\includegraphics[scale=0.7]{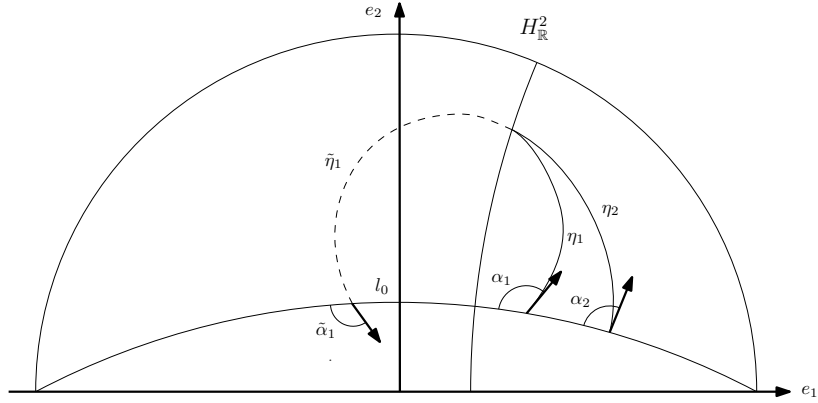}
\caption{The curvature comparison.}
\label{fig:curvature}
\end{figure}
\begin{lemma}[$\kappa(C_t)$ comparison lemma]\label{lem:C_comparison}Let $\eta_1$, $\eta_2$ as in Lemma \ref{lem:k_comparison}. Then, for any two points $\eta_1(t_1)$ and $\eta_2(t_2)$ on the same leaf $\delta_l$, calling $C^1$ and $C^2$ the comparison circles at $\eta_1(t_1)$ and $\eta_2(t_2)$ as in Definition \ref{def:comparison_c}, we have that
\[
\kappa(C^1)\leq\kappa(C^2).
\]
\end{lemma}
\begin{proof}For $i=1,2$, the hyperbolic radius of $C^i$ together with $e_1$ and the geodesic starting from $\eta_i(t_i)$ and hitting $e_1$ perpendicularly bound a geodesic triangle $\triangle_i$. Let $d_1^i$ be the hyperbolic radius, $d_2^i$ be the side touching $e_1$ and $d_3^i$ the the remaining side of $\triangle_i$. Similarly,  for $i=1,2$ and $j=1,2,3$, call $\beta_j^i$ the angle opposite to $d_j^i$, and $\ell_i^j$ the length of $d^i_j$. We refer to Figure \ref{fig:curvatureC}. By construction $\beta_1^1=\beta_1^2=\frac{\pi}{2}$, $\beta_2^i=\alpha_i$, and since $\eta_1(t_1)$ and $\eta_2(t_2)$ are in the same hypercycle by hypothesis, we get $\ell_3^1=\ell_3^2$. Then, by the hyperbolic law of cosines and by Lemma \ref{lem:k_comparison} we get
\[
\kappa(C^1)=\coth(\ell_1^1)=\frac{\cos(\alpha_1)}{\tanh(\ell_3^1)}\leq\frac{\cos(\alpha_2)}{\tanh(\ell_3^2)}=\coth(\ell_1^2)=\kappa(C^2).
\]
\end{proof}
\begin{figure}[htbp]
\centering
\includegraphics[scale=0.7]{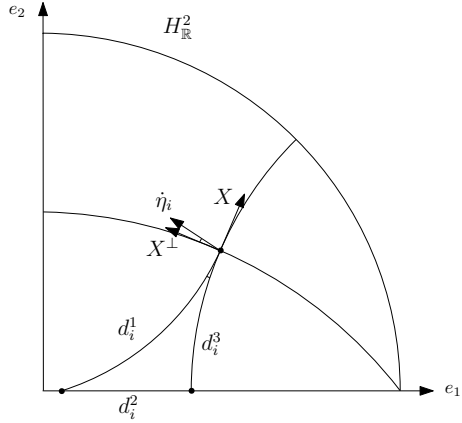}
\caption{The curvature comparison for $C_t$.}
\label{fig:curvatureC}
\end{figure}
\begin{lemma}[$H_1$ comparison lemma]\label{lem:H_comparison}Let $\eta_1$, $\eta_2$ as in Lemma \ref{lem:k_comparison} and let $\tilde\eta_1$ be the reflection of $\eta_1$ with respect to the geodesic passing through $\eta_1(A_1)$ and crossing $e_1$ perpendicularly. Reverse its parametrization, so that the angle that $\dot{\tilde\eta}_1$ makes with $X^\perp$ is equal to $\tilde\alpha_1=-\alpha_1$. Moreover, suppose that 
\[
g_H(N,\dot\eta_2)\leq 0.
\]
Denote the unitary outward pointing normals to $\eta_1$ and $\eta_2$ by $\tilde\nu_1$ and $\nu_2$. Then, for any two points $\tilde\eta_1(t_1)$ and $\eta_2(t_2)$ on the same leaf $\delta_l$ we have that
\[
g_H(N(\tilde\eta_1(t_1)),\tilde\nu_1(t_1))\leq g_H(N(\eta_1(t_2)),\nu_2(t_2)),
\]
with equality if and only if $\dot\eta_2(t_2)$ and $\dot{\tilde\eta}_1(t_1)$ are tangent to the same circle centered in the origin and $\dot{\tilde\eta}_1(t_1)=-\dot{\eta}_2(t_2)$.
\end{lemma}
\begin{proof}
Let $\vartheta(s)$ be the angle that $N$ makes with $X^\perp$ at $\delta_l(s)$ and $s_2<s_1$ be such that $\delta_l(s_1)=\tilde\eta_1(t_1)$ and $\delta_l(s_2)=\eta_1(t_2)$. Setting $\Theta_1:=\vartheta(s_1)-\tilde\alpha_1-\frac\pi 2$ and $\Theta_2:=\vartheta(s_2)-\alpha_2-\frac\pi 2$, we get that
\[
g_H(N(\tilde\eta_1(t_1)),\tilde\nu_1(t_1))=\cos(\Theta_1),
\]
and
\[
g_H(N(\eta_2(t_2)),\nu_2(t_2))=\cos(\Theta_2).
\]
Let $s_2^*$ such that the unit vector at $\delta_l(s_2^*)$ that forms an angle of $\alpha_2$ with $X^\perp$ is tangent to a circle centered in the origin. The value $s_2^*$ must exist and it is greater than $s_1$ because by Lemma \ref{lem:center_C}, Equation \eqref{eq:ANGLES}, we have that $\vartheta(s_2)\geq\alpha_2$. Now, on the intersection of $\delta_l$ with $e_1$ we have $\vartheta=0$, hence by continuity there must be a point $s^*_2$ between this intersection and $\delta_l(s_1)$ such that $\vartheta(s_2^*)=\alpha$.  Then 
\[
\Theta_2=\vartheta(s_2)-\vartheta(s_2^*)=\int_{s_2^*}^{s_2}\dot\vartheta(s)\,ds.
\]
Let $s_1^*=-s_2^*$, and notice that $\vartheta(-s)=\pi-\vartheta(s)$ for every $s\geq0$. Then,
\begin{align*}
\Theta_1&=\vartheta(s_1)-\tilde\alpha_1-\frac\pi 2=\vartheta(s_2)+\alpha_2-(\tilde\alpha_1+\alpha_2)-\frac\pi 2\\
&=\vartheta(s_1)+\vartheta(s_2^*)-\pi-(\tilde\alpha_1+\alpha_2)\\
&=\vartheta(s_1)-\vartheta(-s_2^*)-(\tilde\alpha_1+\alpha_2)\\
&=-\int_{s_1}^{s^*_1}\dot\vartheta(s)\,ds-(\tilde\alpha_1+\alpha_2).
\end{align*}
Hence,
\[
\Theta_2-\Theta_1=(\tilde\alpha_1+\alpha_2)+\int_{s_2^*}^{s_2}\dot\vartheta(s)\,ds+\int_{s_1}^{s^*_1}\dot\vartheta(s)\,ds\leq 0,
\]
since $\dot\vartheta(s)<0$ and $\tilde\alpha_1+\alpha_2\leq 0$ by Lemma \ref{lem:k_comparison}. The equality is attained only when $s_2^*=s_2$, $s_1^*=s_1$ and $\alpha_2=\alpha_1$ as predicted.
\end{proof}
We can prove the main result about the lower curve.
\begin{proposition}\label{prop:down}It $\gamma$ is not a circle centered in the origin, the lower curve is contained in $[a_0,a)$, that is $0<a_1<a$. Furthermore, $a_1\in I_L$ and $\dot\gamma(a_1)=-X(\gamma(a_1))$.
\begin{proof}By property \ref{itm:I_L_big}. of Lemma \ref{lem:preliminary_prop}, $a_1>a>0$. Suppose by contradiction that $a_1=a$. Set $\tilde\eta_1$ to be the (reparametrized) lower curve and $\eta_2$ the upper curve. Choose any point $t\in I_L$ with corresponding $\bar t\in I_U$. Applying Lemma \ref{lem:C_comparison} and Lemma \ref{lem:H_comparison} to $\tilde\eta_1(t)$ and $\eta_2(\bar t)$, and taking advantage of the expression for $\Hc_f$ given in Lemma \ref{lem:const}, we get that
\begin{align*}
\Hc_f(\gamma(\bar t))&=\Hc_f(\gamma(t))=\kappa_{\gamma}(t)+(n-2)\kappa(C_t)+h'(d_{H}(0,\gamma(t)))g_H(\nu(t),N)\\
&<\kappa_{\gamma}(t)+(n-2)\kappa(C_{\bar t})+h'(d_{H}(0,\gamma(t)))g_H(\nu(t),N).
\end{align*}
We have that
\[
d_{H}(0,\gamma(t))\leq d_{H}(0,\gamma(\bar t)).
\]
This can be verified again via the trigonometric rules for hyperbolic triangles: fix $t\in I_U$, and call $\beta$ and $\bar\beta$ the angle that $N$ makes with $X(\gamma(t))$ and $X(\gamma(\bar t))$ respectively. Notice that $0\leq \beta\leq\bar\beta$. Then, calling $d$ the distance of $\gamma(t)$ and $\gamma(\bar t)$ from $e_1$, we get that
\[
\tanh(d_H(0,\gamma(t)))=\frac{\tanh(d)}{\cos(\beta)}\leq \frac{\tanh(d)}{\cos(\bar\beta)}=\tanh(d_H(0,\gamma(\bar t))).
\]
Hence
\[
\Hc_f(\gamma(\bar t))< \kappa_\gamma(t)+(n-2)\kappa(C_{\bar t})+h'(d_H(0,\gamma(\bar t)))g_H(\nu(\bar t),N),
\]
implying
\[
\kappa_{\gamma}(\bar t)<\kappa_{\gamma}(t),
\]
since $h$ is strictly convex. Again, Lemma \ref{lem:k_comparison} tells us that the lower curve hits the $e_1$ axis with an angle strictly smaller than $-\frac\pi2$. Contradiction. Therefore, $a_1<a$. Since $\gamma$ is smooth in $a_1<a$, and the conditions on $I_L$ are closed, we deduce that $a_1\in I_L$. Suppose now that $\alpha(a_1)$ is strictly in the III quadrant. Since $a_1\in I_L$, we can apply again the comparison lemmas to $\gamma(a_1)$ and $\gamma(\bar a_1)$ to infer
\[
\kappa_\gamma(\bar a_1)<\kappa_\gamma(a_1).
\]
By continuity of $\kappa_\gamma$ and $\dot \gamma$ around $a_1$, we get that there exists a neighbourhood of $a_1$ in which $\dot\gamma$ is in the III quadrant and the above inequality holds in the not strict sense. But this implies that $a_1$ is not the supremum of $I_L$. Therefore, the velocity vector of $\gamma$ at $a_1$ has to be equal to $-X$.
\end{proof}
\end{proposition}
We prove the last part of the tangent lemma.
\begin{proposition}\label{prop:curl}If $\gamma$ is not a centered circle, then there exists $0<a_1<a_2$ such that $\dot\gamma(a_2)=X(\gamma(a_2))$.
\end{proposition}
\begin{proof}
If $a_2$ exists we are done. Otherwise, we show that the non existence contradicts the mean-curvature convexity of $\Omega$. Let 
\[
I_c:=\{t\in[a_1,a):\dot\gamma\text{ is in the I or IV quadrant}\}.
\]
Here the index stands for \emph{curl curve}. Set $\tilde a_2:=\sup I_c$. Since $\kappa(a_1)>1$ we have that $a_1<\tilde a_2$. If $\tilde a_2<a$, then the mean convexity of $\Omega$ implies that
\[
\kappa_{\gamma}(\tilde a_2)\geq(n-1)-(n-2)\kappa(C_{\tilde a_2})>0.
\]
To see this, move $\gamma(\tilde a_2)$ on $e_2$ as in Lemma \ref{lem:preliminary_prop}. Then, $C_{\tilde a_2}$ is oriented clockwise, and hence has negative curvature. But this implies that we can extend $I_c$ after $\tilde a_2$, contradicting the definition of $\tilde a_2$. So, we need to rule out the situation in which $\tilde a_2=a$. If it is the case, then again for mean-convexity one has that in $I_c$
\[
\kappa_\gamma(t)>1.
\]
Moreover, for $t\in I_c\setminus\{a_1\}$ we have that $\dot\gamma$ lies in the IV quadrant, because otherwise this implies together with $\kappa_\gamma(t)>0$ that $\gamma$ cannot close at $e_1$. Therefore $\alpha(t)$ lies in the IV quadrant and it is strictly increasing, implying that
\[
\lim_{t\to a^+}\alpha(t)<-\frac\pi2.
\]
This cannot be because of the before mentioned regularity properties of isoperimetric sets.
\end{proof}
The proof of the tangent lemma is then the collection of the results we showed in this section.
\begin{proof}[Proof of Lemma \ref{lem:tangent}]The existence of the chain $0<a_0<a_1<a_2<a$ is ensured by Proposition \ref{prop:up}, Proposition \ref{prop:down} and Proposition \ref{prop:curl}.
\end{proof}

%% file: hopf_sets2.tex
Consider any rank one symmetric space of non-compact type $(M^n,g)=(H_{\K}^m,g)$, $\K\in\{\C,\Ham,\Oct\}$. Set $d=\dim(\K)\in\{2,4,8\}$ so that the real dimension of $M$ is $n=md$. Recall that if $\K=\Oct$, we only have the Cayley plane $H^2_\Oct$. As classical references on symmetric spaces we cite the books of Eberlein \cite{Eberlein} and Helgason \cite{Helgason}. Fix an arbitrary base point $o\in M$, and let $N$ be the unit-length, radial vector field emanating from it. As in Definition \ref{def:hopf_symmetric}, let $\Dist$ be the distribution on $M\setminus\{o\}$ induced by the $(-4)$-eigenspace of the Jacobi operator $R(\cdot,N)N$.  Denote with $\Vist$ the orthogonal complement of $\Dist$ with respect to $g$. For every $x\in M\setminus\{o\}$, we have the orthogonal splitting 
\begin{equation*}
T_xM=\Dist_x\oplus\Vist_x,
\end{equation*}
with orthogonal projections $(\cdot)^\Dist$ and $(\cdot)^\Vist$.  Let now $(\bar M^n,g_H)=(H^n_\R,g_H)$, and choose an arbitrary base point $\bar o$ in it. The isometric identification of $T_oM$ with $T_{\bar o}{\bar M}$ according to the flat metrics 
$(\exp_o^M)^*g\vert_0$ and $(\exp_{\bar o}^{\bar M})^*g_H\vert_0$, induces a well defined diffeomorphism
\[
\Psi=\exp_{\bar o}^{\bar M}\circ(\exp_o^M)^{-1}:M\to \bar M.
\]
With a slight abuse of notation, we still denote with $g_H$ the metric $\Psi^*g_H$, that makes $M$ isometric to $\bar M$. The following lemma allows us to compare $g$ with $g_H$.
\begin{lemma}\label{lem:comparing_manifolds}For every $x\in M\setminus\{o\}$, the splitting
\[
T_xM=\Dist_x\oplus\Vist_x,
\]
is orthogonal with respect to $g_H$. In particular, letting $d_H$ be the Riemannian distance induced by $g_H$ on $M$, one has that
\begin{equation}\label{eq:relation}
g(v,w)= \cosh^2(d_H(o,x))g_H(v^\Dist,w^\Dist)+g_H(v^\Vist,w^\Vist),
\end{equation}
for all $v,w\in T_xM$.
\end{lemma}
\begin{proof}Fix an arbitrary unit direction $N_o\in T_oM$, and let $V_o\in T_oM$ be any vector orthogonal to it with respect to $g\vert_o=g_H\vert_o$. Since the radial geodesics emanating from $o$ are the same for $g$ and $g_H$, the Jacobi field $Y(t)$ along the geodesic $\sigma:t\mapsto\exp_o^M(tN_o)$, determined by the initial conditions $Y(0)=0$, $\dot Y(0)=V_o$ is the same for both metrics. Let $V(t)$ and $V_H(t)$ be the parallel transport of $V_o$ along $\sigma$ with respect to $g$ and $g_H$, respectively. By the very definition of symmetric spaces, the curvature tensor $R$ is itself parallel along geodesics. This implies that
\[
\sinh(t)V_H(t)=Y(t)=\frac{\sinh(\sqrt{-\kappa}t)}{\sqrt{-\kappa}}V(t),
\]
provided $V_o$ belongs to the $\kappa$-eigenspace of the Jacobi operator $R(\cdot,N_o)N_o$. Therefore, parallel vector fields in the eigenspaces are collinear for the two metrics. Hence, for $t>0$ the linear subspaces $\Dist_{\sigma(t)}$ and $\Vist_{\sigma(t)}$ are nothing else than the parallel transport of the corresponding eigenspaces of $R(\cdot,N_o)N_o$ along $\sigma$. It follows that the splitting $T_xM=\Dist_x\oplus \Vist_x$ is orthogonal not only with respect to $g$, but also with respect to the hyperbolic metric $g_H$. Equation \eqref{eq:relation} is a direct consequence of this fact and the definition of the distribution $\Dist$.
\end{proof}
We can now prove Theorem \ref{thm:2}.
\begin{proof}[Proof of Theorem \ref{thm:2}]
    Let $E\subset M$ be an Hopf-symmetric set with outward pointing normal vector field $\nu$ with respect to $g$. By the very definition of being Hopf-symmetric, $\nu^\Dist\equiv 0$. Therefore, thanks to Lemma \ref{lem:comparing_manifolds}, $\nu$ is orthonormal to $\partial E$ also with respect to $g_H$. Let $\vol$ and $\vol_H$ the volume forms associated to $g$ and $g_H$. We have that
    \[
    P(E)=\int_{\partial E}\iota_\nu\vol=\int_{\partial E}\cosh^{d-1}(d_H(o,x))\iota_\nu\vol_H(x),
    \]
    where $\iota:\Omega(M)^p\to\Omega(M)^{p-1}$ denotes the interior product $\iota_X\alpha(\cdot)=\alpha(X,\cdot)$. The volume of $E$ is given by the formula
    \[
    V(E)=\int_E\vol=\int_{E}\cosh^{d-1}(d_H(o,x))\vol_H(x).
    \]
    Hence, the volume and perimeter of Hopf-symmetric sets in $M$ correspond to the volume and perimeter of $\Psi(E)$ in $H^n_\R$ with density equal to $h(r)=\cosh^{d-1}(r)$, concluding the proof.
\end{proof}